\documentclass[a4,12pt]{amsart}
\usepackage{amssymb,amstext,amsmath,amscd,amsthm,amsfonts,enumerate,latexsym}
\usepackage{color}
\usepackage[dvipdfmx]{graphicx}
\usepackage[all]{xy}
\theoremstyle{plain}
\newtheorem{thm}{Theorem}[section]

\newtheorem*{thm*}{Theorem}
\newtheorem*{cor*}{Corollary}

\newtheorem{prop}[thm]{Proposition}
\newtheorem{proposition}[thm]{Proposition}
\newtheorem{lemma}[thm]{Lemma}

\newtheorem{cor}[thm]{Corollary}
\newtheorem{corollary}[thm]{Corollary}
\newtheorem{claim}{Claim}
\newtheorem*{claim*}{Claim}

\theoremstyle{definition}

\newtheorem{definition}[thm]{Definition}
\newtheorem{ex}[thm]{Example}

\newtheorem{remark}[thm]{Remark}

\theoremstyle{remark}

\numberwithin{equation}{thm}

\def\GCD{\operatorname{GCD}}
\def\Ext{\operatorname{Ext}}

\def\bbM{\mathbb{M}}
\def\bbI{\mathbb{I}}

\def\rank{\mathrm{rank}}

\def\m{\mathfrak m}
\def\n{\mathfrak n}

\newcommand{\rma}{\mathrm{a}}

\newcommand{\rme}{\mathrm{e}}

\newcommand{\rmr}{\mathrm{r}}

\newcommand{\rmv}{\mathrm{v}}

\newcommand{\rmJ}{\mathrm{J}}
\newcommand{\rmK}{\mathrm{K}}

\newcommand{\rmQ}{\mathrm{Q}}

\newcommand{\calX}{\mathcal{X}}

\newcommand{\fka}{\mathfrak{a}}
\newcommand{\fkb}{\mathfrak{b}}
\newcommand{\fkc}{\mathfrak{c}}

\newcommand{\fkm}{\mathfrak{m}}
\newcommand{\fkn}{\mathfrak{n}}

\newcommand{\mapright}[1]{%
\smash{\mathop{%
\hbox to 1cm{\rightarrowfill}}\limits^{#1}}}

\newcommand{\mapleft}[1]{%
\smash{\mathop{%
\hbox to 1cm{\leftarrowfill}}\limits_{#1}}}

\def\Syz{\mathrm{Syz}}

\def\GGL{\operatorname{GGL}}

\def\gr{\mbox{\rm gr}}

\title[Chains of Ulrich ideals in Cohen-Macaulay local rings of dimension one]{The structure of chains of Ulrich ideals in Cohen-Macaulay local rings of dimension one}

\author{Shiro Goto}
\address{Department of Mathematics, School of Science and Technology, Meiji University, 1-1-1 Higashi-mita, Tama-ku, Kawasaki 214-8571, Japan}
\email{shirogoto@gmail.com}

\author{Ryotaro Isobe}
\address{Department of Mathematics and Informatics, Graduate School of Science and Technology, Chiba University, Chiba-shi 263, Japan}
\email{r.isobe.math@gmail.com}

\author{Shinya Kumashiro}
\address{Department of Mathematics and Informatics, Graduate School of Science and Technology, Chiba University, Chiba-shi 263, Japan}
\email{polar1412@gmail.com}

\thanks{The first author was partially supported by the JSPS Grant-in-Aid for Scientific Research (C) 16K05112. The second and third authors were partially supported by Birateral Programs (Joint Research) of JSPS and International Research Supporting Programs of Meiji University.}

\begin{document}
\maketitle

\setlength{\baselineskip}{20pt}

\begin{abstract}
This paper studies Ulrich ideals in one-dimensional Cohen-Macaulay local rings. A correspondence between Ulrich ideals and overrings is given. Using the correspondence, chains of Ulrich ideals are closely explored. The specific cases where the rings are of minimal multiplicity and $\GGL$ rings are analyzed.   
\end{abstract}

%{\footnotesize \tableofcontents}

%%%%%%%%%%%%%%%%%%%%%%%%%%%%%%%%%%%%%%%%%%%%%%%%%%%%%%%%%%%%%%%%%%%%%%%%%%%%%%%%%%%%%%%%%%%%%%%%%%%%%%%%%%%%%%%%%%%%%%%%%%%%%%%%%%%%%%%%%%%%%%%%%%%%%%%%%%%%%%%%%%%%%%%%%%%%%%%%%%%%%%%%%%%%%%%%%%%%%%%%%%%%%%%%%%%%%%%%

\section{Introduction}
The purpose of this paper is to investigate the behavior of chains of Ulrich ideals, in a one-dimensional Cohen-Macaulay local ring, in connection with the structure of birational finite extensions of the base ring.

The notion of Ulrich ideals is a generalization of stable maximal ideals, which dates back to 1971, when the monumental paper \cite{L} of J. Lipman was published. The modern treatment of Ulrich ideals was started by \cite{GOTWY1, GOTWY2} in 2014, and has been explored in connection with the representation theory of rings. In \cite{GOTWY1}, the basic properties of Ulrich ideals are summarized, whereas in \cite{GOTWY2},  Ulrich ideals in two-dimensional Gorenstein rational singularities are closely studied with a concrete classification. However, in contrast to the existing research on Ulrich ideals, the theory pertaining to the one-dimensional case does not seem capable of growth. Some part of the theory, including research on the ubiquity as well as the structure of the chains of Ulrich ideals, seems to have been left unchallenged. In the current paper, we focus our attention on the one-dimensional case, clarifying the relationship between Ulrich ideals and the birational finite extensions of the base ring. The main objective is to understand the behavior of chains of Ulrich ideals in one-dimensional Cohen-Macaulay local rings.

To explain our objective as well as our main results, let us begin with the definition of Ulrich ideals. Although we shall focus our attention on the one-dimensional case, we would like to state the general definition, in the case of any arbitrary dimension. Let $(R,\m)$ be a Cohen-Macaulay local ring with $d= \dim R \ge 0$.

\begin{definition}[\cite{GOTWY1}]\label{def1.1}
Let $I$ be an $\m$-primary ideal of $R$ and assume that $I$ contains a parameter ideal $Q=(a_1,a_2, \ldots, a_d)$ of $R$ as a reduction. We say that $I$ is an $Ulrich\ ideal$ of $R$, if the following conditions are satisfied.\begin{enumerate}[{\rm (1)}]
\item $I\ne Q$,
\item $I^2=QI$, and 
\item $I/I^2$ is a free $R/I$-module.
\end{enumerate}
\end{definition}
We notice that Condition (2) together with Condition (1) are equivalent to saying that the associated graded ring $\gr_I(R) = \bigoplus_{n \ge 0}I^n/I^{n+1}$ of $I$  is a Cohen-Macaulay ring and $\rma(\gr_I(R)) = 1-d$, where $\rma(\gr_I(R))$ denotes the a-invariant of $\gr_I(R)$. Therefore, these two conditions are independent of the choice of reductions $Q$ of $I$. In addition, assuming Condition (2) is satisfied, Condition (3) is equivalent to saying that $I/Q$ is a free $R/I$-module (\cite[Lemma 2.3]{GOTWY1}). We also notice that Condition (3) is automatically satisfied if $I = \fkm$, so that the maximal ideal $\fkm$ is an Ulrich ideal of $R$ if and only if $R$ is not a regular local ring, possessing minimal multiplicity (\cite{S}).  From this perspective, Ulrich ideals are a kind of generalization of stable maximal ideals, which Lipman \cite{L} started to analyze in 1971. 

Here, let us briefly summarize some basic properties of Ulrich ideals, as seen in \cite {GOTWY1, GTT2}. Although we need only a part of them, let us also include some superfluity in order to show what specific properties Ulrich ideals enjoy. Throughout this paper, let $\rmr(R)$ denote the Cohen-Macaulay type of $R$, and let $\Syz_R^i(M)$ denote, for each integer $i \ge 0$ and for each finitely generated $R$-module $M$, the $i$-th syzygy module of $M$ in its minimal free resolution.

\begin{thm}[\cite{GOTWY1, GTT2}]
Let $I$ be an Ulrich ideal of a Cohen-Macaulay local ring $R$ of dimension $d \ge 0$ and set $t = n-d ~(>0)$, where $n$ denotes the number of elements in a minimal system of generators of $I$. Let 
$$\cdots \to F_i \overset{\partial_i}{\to} F_{i-1} \to \cdots \to F_1 \overset{\partial_1}{\to}  F_0 = R \to R/I \to 0$$
be a minimal free resolution of $R/I$. Then $\rmr(R) = t{\cdot}\rmr(R/I)$ and the following assertions hold true.
\begin{enumerate}[$(1)$]
\item  $\mathbf{I}(\partial_i) = I$ for $i \ge 1$.
\item  For $i \ge 0$, $\beta_i = 
\begin{cases}
t^{i-d}{\cdot}(t+1)^d & \ \ (i \ge d),\\
\binom{d}{i}+ t{\cdot}\beta_{i-1} & \ \ (1 \le i \le d),\\
1  & \ \  (i=0).
\end{cases}$
\item  $\Syz_R^{i+1}(R/I) \cong [\Syz_R^i(R/I)]^{\oplus t}$ for $i \ge d$. 
\item  For $i \in \Bbb Z$,  $\Ext_R^i(R/I, R)\cong
\begin{cases}
(0) & (i < d),\\
(R/I)^{\oplus t} & (i = d),\\
(R/I)^{\oplus (t^2-1){\cdot}t^{i - (d+1)} } & (i > d). 
\end{cases}$
\end{enumerate}
Here $\mathbf{I}(\partial_i)$ denotes the ideal of $R$ generated by the entries of the matrix $\partial_i$, and $\beta_i = \rank_RF_i$. 
\end{thm}

Because Ulrich ideals are a very special kind of ideals, it seems natural to expect that, in the behavior of Ulrich ideals, there might be contained ample information on base rings, once they exist. As stated above, this is the case of two-dimensional Gorenstein rational singularities, and the present objects of study are rings of dimension one.

In what follows, unless otherwise specified, let $(R,\fkm)$ be a Cohen-Macaulay local ring with $\dim R = 1$.  Our main targets are chains $I_n \subsetneq I_{n-1} \subsetneq \cdots \subsetneq I_1$~($n \ge 2$) of Ulrich ideals in $R$.  Let $I$ be an Ulrich ideal of $R$ with a reduction $Q=(a)$. We set $A = I:I$ in the total ring  of fractions of $R$. Hence, $A$ is a birational  finite extension of $R$, and $I = aA$. Firstly, we study the close connection between the structure of the ideal $I$ and the $R$-algebra $A$. Secondly, let $J$ be an Ulrich ideal of $R$ and assume that $I \subsetneq J$. Then, we will show that $\mu_R(J) = \mu_R(I)$, where $\mu_R(*)$ denotes the number of elements in a minimal system of generators, and that $J = (b) + I$ for some $a,b \in \fkm$ with $I = abA$. Consequently, we have the following, which is one of the main results of this paper.

\begin{thm}\label{1.1}
Let $(R,\m)$ be a Cohen-Macaulay local ring with $\dim R=1$.
Then the following assertions hold true.
\begin{enumerate}[{\rm (1)}]
 \item Let $I$ be an Ulrich ideal of $R$ and $A = I:I$. Let $a_1,a_2, \ldots, a_n\in\m$~$($$n \ge 2$$)$ and assume that $I = a_1a_2\cdots a_nA$. For $1\leq i\leq n$, let $I_i=(a_1a_2 \cdots a_i)+I$. Then each $I_i$ is an Ulrich ideal of $R$ and $$I=I_n\subsetneq I_{n-1}\subsetneq \cdots \subsetneq I_1.$$ 
\item Conversely, let $I_1, I_2, \ldots, I_n$~$($$n \ge 2$$)$ be Ulrich ideals of $R$ and suppose that $$I_n\subsetneq I_{n-1}\subsetneq \cdots \subsetneq I_1.$$ We set $I = I_n$ and $A=I:I$. Then there exist elements $a_1,a_2, \ldots, a_n\in\m$ such that $I = a_1a_2\cdots a_nA$ and $I_i=(a_1a_2\cdots a_i)+I$ for all $1\leq i\leq n-1$. 
\end{enumerate}
\end{thm}

Let $I$ and $J$ be Ulrich ideals of $R$ and assume that $I \subsetneq J$. We set $B = J:J$. Let us write $J = (b) + I$ for some $b \in \fkm$. We then have that $J^2 = bJ$ and that $B$ is a local ring with the maximal ideal $\fkn = \fkm + \displaystyle\frac{I}{b}$, where $\displaystyle\frac{I}{b} = \left\{\displaystyle\frac{i}{b} \mid i \in I\right\}~(=b^{-1}I)$. We furthermore have the following.

\begin{thm}\label{1.2}
$\displaystyle\frac{I}{b}$ is an Ulrich ideal of the Cohen-Macaulay local ring $B$ of dimension one and there is a one-to-one correspondence   $\fka \mapsto \displaystyle\frac{\fka}{b}$ between the Ulrich ideals $\fka$ of $R$ such that $I \subseteq \fka \subsetneq J$ and the Ulrich ideals $\fkb$ of $B$ such that $\displaystyle\frac{I}{b} \subseteq \fkb$.
\end{thm}

These two theorems convey to us that the behavior of chains of Ulrich ideals in a given one-dimensional Cohen-Macaulay local ring could be understood via the correspondence, and the relationship between the structure of Cohen-Macaulay local rings $R$ and $B$ could be grasped through the correspondence, which we shall closely discuss in this paper.

We now explain how this paper is organized. In Section 2, we will summarize some preliminaries, which we shall need later to prove the main results. The proof of Theorems \ref{1.1} and \ref{1.2} will be given in Section 3. In Section 4, we shall study the case  where the base rings $R$ are not regular but possess minimal multiplicity (\cite{S}), and show that the set of Ulrich ideals of $R$ are totally ordered with respect to inclusion. In Section 5, we explore the case where $R$ is a $\GGL$ ring (\cite{GK}).

In what follows, let $(R,\fkm)$ be a Cohen-Macaulay local ring with $\dim R = 1$. Let $\rmQ(R)$ (resp. $\calX_R$) stand for  the total ring of fractions of $R$ (resp. the set of all the Ulrich ideals in $R$). We denote by $\overline{R}$, the integral closure of $R$ in $\rmQ(R)$. For a finitely generated $R$-module $M$, let $\mu_R(M)$ (resp. $\ell_R(M)$) be the number of elements in a minimal system of generators (resp. the length) of $M$. For each $\m$-primary ideal $\fka$ of $R$, let $$\rme_\fka^0(R) = \underset{n \to \infty}{\lim}\frac{\ell_R(R/\fka^{n})}{n}$$ stand for the multiplicity of $R$ with respect to $\fka$. By $\rmv(R)$ (resp. $\rme(R))$ we denote the embedding dimension $\mu_R(\fkm)$ of $R$ (resp. $\rme_\m^0(R)$). Let $\widehat{R}$ denote the $\m$-adic completion of $R$.

%%%%%%%%%%%%%%%%%%%%%%%%%%%%%%%%%%%%%%%%%%%%%%%%%%%%%%%%%%%%%%%%%%%%%%%%%%%%%%%%%%%%%%%%%%

%%%%%%%%%%%%%%%%%%%%%%%%%%%%%%%%%%%%%%%%%%%%%%%%%%%%%%%%%%%%%%%%%%%%%%%%%%%%%%%%%%%%%%%%%%

\section{Preliminaries}\label{section2}

Let us summarize preliminary facts on $\fkm$-primary ideals of $R$, which we need throughout this paper.

In this section, let $I$ be an $\fkm$-primary ideal of $R$, for which we will assume Condition (C) in Definition \ref{2.2} to be satisfied. This condition is a partial extraction from Definition \ref{def1.1} of Ulrich ideals; hence  every Ulrich ideal satisfies it (see Remark \ref{2.3}).

Firstly, we assume that $I$ contains an element $a \in I$ with $I^2=aI$. We set $A=I:I$ and $$\displaystyle\frac{I}{a} =\left\{\displaystyle\frac{x}{a} \mid x \in I\right\}=a^{-1}I$$ in $\rmQ(R)$. Therefore, $A$ is a birational finite extension of $R$ such that $R \subseteq A\subseteq \overline{R}$, and  $A = \displaystyle\frac{I}{a}$, because $I^2 = aI$; hence $I = aA$. We then  have the following.

\begin{prop}\label{2.1}  If $I=(a):_RI$, then $A=R:I$ and $I = R : A$, whence $R:(R:I)=I$.
\end{prop}

\begin{proof}
Notice that $I = (a):_RI = (a) : I=a[R:I]$ and we have $A=R:I$, because $I = aA$. We get $R:A=I$, since $R:A= R:\displaystyle\frac{I}{a}= a[R:I] = aA$. 
\end{proof}

Let us now give the following.

\begin{definition}\label{2.2}
Let $I$ be an $\m$-primary ideal of $R$ and set $A=I:I$. We say that $I$ satisfies Condition (C), if 
\begin{enumerate}[{\rm (i)}]
\item $A/R \cong (R/I)^{\oplus t}$ as an $R$-module for some $t>0$, and
\item $A=R:I$.  
\end{enumerate}
\end{definition}
Consequently, $I=R:A$ by Condition (i), when $I$ satisfies Condition (C).

\begin{remark}\label{2.3}
Let $I \in \calX_R$. Then $I$ satisfies Condition (C). In fact, choose $a \in I$ so that $I^2 = aI$. Then, $I/(a) \cong (R/I)^{\oplus t}$ as an $R/I$-module, where $t = \mu_R(I)-1 > 0$ (\cite[Lemma 2.3]{GOTWY1}). Therefore, $I = (a):_RI$, so that $I$ satisfies the hypothesis in Proposition \ref{2.1}, whence $A = R:I$.  Notice that $A/R \cong I/(a) \cong (R/I)^{\oplus t}$, because $I = aA$.
\end{remark}

We assume, throughout this section, that our $\m$-primary ideal $I$ satisfies Condition (C). We choose elements $\{f_i\}_{1 \le i \le t}$ of $A$ so that $$A=R+\sum_{i=1}^tRf_i.$$ Therefore, the images $\{\overline{f_i}\}_{1 \le i \le t}$ of $\{f_i\}_{1 \le i \le t}$ in $A/R$ form a free basis of the $R/I$-module $A/R$. We then have the following.

\begin{lemma}\label{2.4}
$aA\cap R\subseteq (a)+I$ for all $ a\in R$.
\end{lemma}
\begin{proof}
Let $x \in aA\cap R$ and write $x = ay$ with $y \in A$. We write $y=c_0 + \sum_{i=1}^tc_if_i$ with $c_i \in R$. Then, $ac_i \in I$ for $1 \le i \le t$, since $x = ac_0 + \sum_{i=1}^t (ac_i)f_i \in R$. Therefore, $(ac_i)f_i \in IA = I$ for all $1 \le i \le t$, so that $x \in (a) +I$ as claimed.
\end{proof}

\begin{cor}\label{2.5}
Let $J$ be an $\fkm$-primary ideal of $R$ and assume that $J$ contains an element $b \in J$ such that $J^2 = bJ$ and $J = (b):_RJ$. If $I \subseteq J$, then $J=(b)+I$.
\end{cor}

\begin{proof}
We set $B=J:J$. Then $B=R:J$ and $J=bB$ by Proposition \ref{2.1}, so that  $B=R:J \subseteq A = R:I$, since $I\subseteq J$. Consequently, $J=bB\subseteq bA\cap R\subseteq (b)+I$ by Lemma \ref{2.4}, whence  $J=(b)+I$.
\end{proof}

In what follows, let $J$ be an $\fkm$-primary ideal of $R$ and assume that $J$ contains an element $b \in J$ such that $J^2 = bJ$ and $J = (b):_RJ$. We set $B=J:J$. Then $B=R:J=\displaystyle\frac{J}{b}$ by Proposition \ref{2.1}. Throughout,  suppose that $I \subsetneq J$. Therefore, since $J=(b)+I$ by Corollary \ref{2.5}, we get $$B=\displaystyle\frac{J}{b} = R+\displaystyle\frac{I}{b}.$$  Let  $\fka=\displaystyle\frac{I}{b}$. Therefore, $\fka$ is an ideal of $A$ containing $I$, so that $\fka$ is also an ideal of $B$ with $$R/(\fka \cap R) \cong B/\fka.$$

With this setting, we have the following.

\begin{lemma}\label{2.6}
The following assertions hold true.
\begin{enumerate}[{\rm (1)}]
\item $A/B\cong (B/\fka)^{\oplus t}$ as a $B$-module.
\item $\fka\cap R=I:_RJ$.
\item $\ell_R([I:_RJ]/I)=\ell_R(R/J)$.
\item $I=[b{\cdot}(I:_RJ)]A$.
\end{enumerate}
\end{lemma}

\begin{proof}
(1) Since $A=R+\sum_{i=1}^tRf_i$, we get $A/B=\sum_{i=1}^tB\overline{f_i}$ where $\overline{f_i}$ denotes the image of $f_i$ in $A/B$. Let $\{b_i\}_{1 \le i \le t}$ be elements of $B=\displaystyle\frac{J}{b}$ and assume that $\sum_{i=1}^tb_if_i \in B$. Then, since $\sum_{i=1}^t(bb_i)f_i \in R$ and $bb_i \in R$ for all $1 \le i \le t$, we have $bb_i\in I$, so that $b_i\in \displaystyle\frac{I}{b}=\fka$. Hence $A/B\cong (B/\fka)^{\oplus t}$ as a $B$-module.

(2) This is standard, because $J=(b)+I$ and $\fka = \displaystyle\frac{I}{b}$.

(3) Since $J/I=[(b)+I]/I\cong R/[I:_RJ]$, we get $$\ell_R([I:_RJ]/I)=\ell_R(R/I)-\ell_R(R/[I:_RJ])=\ell_R(R/I)-\ell_R(J/I)=\ell_R(R/J).$$

(4) We have $[b{\cdot}(I:_RJ)]A\subseteq I$, since $b{\cdot}(I:_RJ)\subseteq I$ and $IA=I$. To see the reverse inclusion, let $x \in I$. Then $x \in J=bB \subseteq bA$. We write $x=b[c_0+\sum_{i=1}^tc_if_i]$ with $c_i\in R$. Then $bc_i \in I$ for $1\leq i\leq t$ since $x \in R$, so that $(bc_i)f_i \in I$ for all $1 \le i\le t$, because $I$ is an ideal of $A$. Therefore, $bc_0\in I$, since $x=bc_0+\sum_{i=1}^t(bc_i)f_i \in I$. Consequently, $c_i\in I:_Rb = I:_RJ$ for all $0\leq i\leq t$, so that $x \in [b{\cdot}(I:_RJ)]A$ as wanted.
\end{proof}

\begin{cor}\label{2.7}
$J/(b)\cong([I:_RJ]/I)^{\oplus t}$ as an $R$-module. Hence $\ell_R(J/(b)) = t{\cdot}\ell_R(R/J)$.
\end{cor}

\begin{proof}
We consider the exact sequence $$0 \to B/R \to A/R \to A/B \to 0$$ of $R$-modules. By Lemma  \ref{2.6} (1), $A/B$ is a free $B/\fka$-module of  rank $t$, possessing the images of $\{f_i\}_{1 \le i \le t}$ in $A/B$ as a free basis. Because $A/R$ is a free $R/I$-module of rank $t$, also possessing the images of $\{f_i\}_{1 \le i \le t}$ in $A/R$ as a free basis, we naturally get an isomorphism between the following two canonical exact sequences; 
$$\xymatrix{
0\ar[r] &B/R\ar[r]^{i}\ar[d]^{\wr}&A/R\ar[r]\ar[d]^{\wr} &A/B\ar[r]\ar[d]^{\wr} &0\\
0\ar[r] &([\fka\cap R]/I)^{\oplus t}\ar[r]^{i}\ar@{}[ur]|{\circlearrowleft} &(R/I)^{\oplus t}\ar[r]\ar[r]\ar@{}[ur]|{\circlearrowleft} &(B/\fka)^{\oplus t}\ar[r] &0
}$$
Since $B/R=\displaystyle\frac{J}{b}/R \cong J/(b)$ and $\fka \cap R=I:_RJ$ by Lemma \ref{2.6} (2), we get $$J/(b)\cong([I:_RJ]/I)^{\oplus t}.$$ The second assertion now follows from Lemma \ref{2.6} (3). 
\end{proof}

The following is the heart of this section.

\begin{proposition}\label{2.8}
The following conditions are equivalent.
\begin{enumerate}[{\rm (1)}]
\item $J\in\mathcal{X}_R$.
\item $\mu_R([I:_RJ]/I)=1$.
\item $[I:_RJ]/I\cong R/J$ as an $R$-module.
\end{enumerate}
When this is the case, $\mu_R(J)=t+1$.
\end{proposition}
\begin{proof}
The implication (3) $\Rightarrow (2)$ is clear, and the reverse implication follows from the equality $\ell_R([I:_RJ]/I)=\ell_R(R/J)$ of Lemma \ref{2.6} (3).

(1) $\Rightarrow$ (3) Suppose that $J\in\mathcal{X}_R$. Then $J/(b)$ is $R/J$-free, so that by Corollary \ref{2.7}, $[I:_RJ]/I$ is a free $R/J$-module, whence $[I:_RJ]/I\cong R/J$ by Lemma \ref{2.6} (3).

(3) $\Rightarrow$ (1) We have $J/(b)\cong ([I:_RJ]/I)^{\oplus t}\cong (R/J)^{\oplus t}$  by Corollary \ref{2.7}, so that by Definition \ref{def1.1}, $J\in\mathcal{X}_R$ with  $\mu_R(J)=t+1$.
\end{proof}

We now come to the main result of this section, which plays a key role in Section 5.

\begin{thm}\label{2.9}
The following assertions hold true.
\begin{enumerate}[{\rm (1)}]
\item Suppose that $J\in\mathcal{X}_R$. Then there exists an element $c \in \m$ such that $I=bcA$. Consequently, $I\in\mathcal{X}_R$ and $\mu_R(I)=\mu_R(J)=t+1$.
\item Suppose that $t \geq 2$. Then $I\in\mathcal{X}_R$ if and only if $J\in\mathcal{X}_R$.
\end{enumerate}
\end{thm}

\begin{proof}
(1) Since $J\in\mathcal{X}_R$,  by Proposition \ref{2.8} we get an element $c \in \m$ such that $I:_RJ=(c)+I$.  Therefore,  by Lemma \ref{2.6} (4) we have $$I=[b{\cdot}(I:_RJ)]A=[b{\cdot}((c)+I)]A=bcA+bIA = bcA + bI,$$ whence $I=bcA$ by Nakayama's lemma. Let $a = bc$. Then $I^2=(aA)^2 = a{\cdot}a A = aI$, so that $(a)$ is a reduction of $I$; hence $A = \displaystyle\frac{I}{a}$. Consequently, $I/(a)\cong A/R\cong(R/I)^{\oplus t}$, so that $I\in\mathcal{X}_R$ with $\mu_R(I)=t+1$. Therefore, $\mu_R(I) = \mu_R(J)$, because  $\mu_R(J)=t+1$ by Proposition \ref{2.8}.

(2) We have only to show the {\em only if} part. Suppose that $I\in\mathcal{X}_R$ and choose $a\in I$ so that $I^2=aI$; hence  $A=\displaystyle\frac{I}{a}$. We then have $\mu_R(I)=t+1$, since $I/(a)\cong A/R\cong (R/I)^{\oplus t}$. Consequently, since $J = (b)+I$, we get 
$$\mu_R(J/(b))=\mu_R([(b)+I]/(b)) \leq \mu_R(I)=t+1.$$
 On the other hand, we have $\mu_R(J/(b))=t{\cdot}\mu_R([I:_RJ]/I)$, because $J/(b)\cong([I:_RJ]/I)^{\oplus t}$ by Corollary \ref{2.7}. Hence $$t{\cdot}(\mu_R([I:_RJ]/I)-1)\leq 1,$$ so that $\mu_R([I:_RJ]/I)=1$ because $t \geq 2$. Thus by Proposition \ref{2.8}, $J \in\mathcal{X}_R$ as claimed.
\end{proof}

%%%%%%%%%%%%%%%%%%%%%%%%%%%%%%%%%%%%%%%%%%%%%%%%%%%%%%%%%%%%%%%%%%%%%%%%%%%%%%%%%%%%%%%%%%%%

\section{Chains of Ulrich ideals}
In this section, we study the structure of chains of Ulrich ideals in $R$. First of all, remember that all the Ulrich ideals of $R$ satisfy Condition (C) stated in Definition \ref{2.2} (see Remark \ref{2.3}), and summarizing the arguments in Section 2, we readily get the following.

\begin{thm}\label{3.1}
Let $I, J \in \mathcal{X}_R$ and suppose that $I\subsetneq J$. Choose $b\in J$ so that $J^2=bJ$. Then the following assertions hold true.
\begin{enumerate}[{\rm (1)}]
\item $J=(b)+I$.
\item $\mu_R(J)=\mu_R(I)$.
\item There exists an element $c\in \m$ such that $I=bcA$, so that $(bc)$ is a reduction of $I$, where $A=I:I$.
\end{enumerate}
\end{thm}

We begin with the following, which shows that Ulrich ideals behave well, if $R$ possesses minimal multiplicity. We shall discuss this phenomenon more closely in Section 4.

\begin{cor}\label{3.2}
Suppose that $\rmv(R)=\rme(R) > 1$ and let $I \in \calX_R$. Then $\mu_R(I)=\rmv(R)$ and $R/I$ is a Gorenstein ring.   
\end{cor}

\begin{proof}
We have $\fkm \in \calX_R$ and $\rmr(R) = \rmv(R)-1$, because $\rmv(R)=\rme(R) > 1$. Hence by Theorem \ref{3.1} (2), $\mu_R(I) = \mu_R(\m)= \rmv(R)$. The second assertion follows from the equality $\rmr(R)=[\mu_R(I)-1]{\cdot}\rmr(R/I)$ (see \cite[Theorem 2.5]{GTT2}).
\end{proof}

For each $I \in \calX_R$, Assertion (3) in Theorem \ref{3.1} characterizes those ideals $J \in \calX_R$ such that $I \subsetneq J$. Namely, we have the following. 

\begin{cor}\label{3.6}
Let $I\in \mathcal{X}_R$. Then
$$ \{ J\in \mathcal{X}_R \ |\ I\subsetneq J \} =\left\{(b)+I \ |\ b \in \m~\text{such~that}~(bc)~\text{is~a~reduction~of}~I~\text{for~some}~c \in \fkm \right\}.$$
\end{cor}

\begin{proof}
Let $b, c \in \m$ and suppose that $(bc)$ is a reduction of $I$. We set $J = (b) +I$. We shall show that $J \in \calX_R$ and $I \subsetneq J$. Because $bc \not\in \m I$, we have $b, c \notin I$, whence $I \subsetneq J$. If $J=(b)$, we then have $I=bcA \subseteq J = (b)$ where $A=I:I$, so that  $cA \subseteq R$. This is impossible, because $c \not\in R:A=I$ (see Lemma \ref{2.1}). Hence, $(b)\subsetneq J$. Because $I^2=bcI$, we have $J^2=bJ+I^2=bJ+bcI=bJ$. Let us check that $J/(b)$ is a free $R/J$-module. Let $\{f_i\}_{1 \le i \le t}$~($t = \mu_R(I) -1> 0$) be elements of $A$ such that $A= R + \sum_{i=1}^tRf_i$, so that their images $\{\overline{f_i}\}_{1 \le i \le t}$ in $A/R$ form a free basis of the $R/I$-module $A/R$ (remember that $I$ satisfies Condition (C) of Definition \ref{2.2}). We then have $$J=(b)+I=(b)+bcA=(b) + \sum_{i=1}^tR{\cdot}(bc)f_i.$$  Let $\{c_i\}_{1 \le i \le t}$ be elements of $R$ and assume that $\sum_{i=1}^t c_i{\cdot}(bcf_i) \in (b)$. Then, since $\sum_{i=1}^tc_ic{\cdot}f_i \in R$, we have $c_ic \in I=bcA$, so that $c_i\in bA \cap R$ for all $1\leq i\leq t$. Therefore, because $bA\cap R\subseteq (b)+I=J$ by Lemma \ref{2.4}, we get $c_i \in J$, whence $J/(b)\cong(R/J)^{\oplus t}$. Thus, $J = (b)+I \in \calX_R$. 
\end{proof}

The equality $\mu_R(I) = \mu_R(J)$ does not hold true in general, if $I$ and $J$ are incomparable, as we show in the following.

\begin{ex}
Let $S=k[[X_1, X_2, X_3, X_4]]$ be the formal power series ring over a field $k$ and consider the  matrix ${\Bbb M}=
\left(\begin{smallmatrix}
X_1 & X_2 & X_3 \\
X_2 & X_3 & X_1 
\end{smallmatrix}\right)
$.
We set $R=S/[\fka + (X_4^2)]$, where $\fka$ denotes the ideal of $S$ generated by the $2 \times 2$ minors of $\bbM$. Let $x_i$ denote the image of $X_i$ in $R$ for each $i=1,2,3, 4$. Then, $(x_1, x_2, x_3)$ and $(x_1, x_4)$ are Ulrich ideals of $R$ with different numbers of generators, and  they are incomparable with respect to inclusion.
\end{ex}

We are now ready to prove Theorem \ref{1.1}.

\begin{proof}[Proof of Theorem \ref{1.1}]

(1) This is a direct consequence of Corollary \ref{3.6}.

(2) By Theorem \ref{3.1}, we may assume that $n >2$ and that our assertion holds true for $n-1$.  Therefore, there exist elements $a_1, a_2, \ldots, a_{n-1} \in \m$ such that $(a_1a_2\cdots a_{n-1})$ is a reduction of $I_{n-1}$ and $I_i=(a_1a_2\cdots a_i)+I_{n-1}$ for all $1\leq i\leq n-2$. Now apply Theorem \ref{3.1} to the chain $I_n \subsetneq I_{n-1}$. We then have $I_{n-1}=(a_1a_2\cdots a_{n-1})+I_n$ together with one more element $a_n\in\m$ so that $(a_1a_2\cdots a_{n-1}){\cdot}a_nA=I_n$. Hence $$I_i=(a_1a_2\cdots a_i)+I_{n-1}=(a_1a_2\cdots a_i)+I_n$$ for all $1\leq i\leq n-1$.  
\end{proof}

In order to prove Theorem \ref{1.2}, we need more preliminaries. Let us begin with the following.

\begin{thm}\label{3.7}
Suppose that $I, J\in \mathcal{X}_R$ and $I\subsetneq J$. Let $b\in J$ such that $J^2=bJ$ and $B = J:J$. Then the following assertions hold true.
\begin{enumerate}[{\rm (1)}]
\item $B=R+\displaystyle\frac{I}{b}$ and $\displaystyle\frac{I}{b}=I:J$.
\item $B$ is a Cohen-Macaulay local ring with $\dim B=1$ and $\n=\m+\displaystyle\frac{I}{b}$ the maximal ideal. Hence $R/\m \cong B/\n$.
\item $\displaystyle\frac{I}{b}\in\mathcal{X}_B$ and $\mu_B(\displaystyle\frac{I}{b})=\mu_R(I)$.
\item $\rmr(B)=\rmr(R)$ and $\rme(B)=\rme(R)$. Therefore, $\rmv(B)=\rme(B)$ if and only if $\rmv(R)=\rme(R)$.
\end{enumerate}
\end{thm}

\begin{proof} We set $A = I:I$. Hence $R \subsetneq B \subsetneq A$ by Proposition \ref{2.1}.  Let $t = \mu_R(I) - 1$.

(1) Because $J=(b)+I$ and $B = \displaystyle\frac{J}{b}$, we get $B=R+\displaystyle\frac{I}{b}$. We have $I:J \subseteq \displaystyle\frac{I}{b}$, since  $b \in J$. Therefore, $\displaystyle\frac{I}{b}=I:J$, because $$J\cdot \displaystyle\frac{I}{b}=I\cdot \displaystyle\frac{J}{b}=IB\subseteq IA=I.$$

(2) It suffices to show that $B$ is a local ring with maximal ideal $\n=\m + \displaystyle\displaystyle\frac{I}{b}$. Let $\fka = \displaystyle\frac{I}{b}$. Choose $c \in \m$ so that $I=bcA$. We then have $\fka =cA \subseteq \fkm A \subseteq \rmJ(A)$, where $\rmJ(A)$ denotes the Jacobson radical of $A$. Therefore, $\n = \m + cA$ is an ideal of $B=R+cA$, and $\fkn \subseteq \rmJ(B)$, because $A$ is a finite extension of $B$. On the other hand, because $R/\m \cong B/\n$, $\fkn$ is a maximal ideal of $B$, so that $(B,\fkn)$ is a local ring.

(3) We have $\fka^2=c\fka$, since $\fka = cA$. Notice that $\fka \ne cB$, since $A \ne B$. Then, because $\fka/cB\cong A/B\cong(B/\fka)^{\oplus t}$ by Lemma \ref{2.6} (1), we get $\fka\in\mathcal{X}_B$ and $\mu_B(\fka)=t+1=\mu_R(I)$.

(4) We set $L=(c)+I$. Then, since $bcA=I$, $L\in\mathcal{X}_R$ and $\mu_R(L)=\mu_R(I)=t+1$ by Corollary \ref{3.6} and Theorem \ref{3.1} (2). Therefore, $\rmr(R)=t{\cdot} \rmr(R/L)$ by \cite[Theorem 2.5]{GTT2}, while $\rmr(B)=t{\cdot}\rmr(B/\fka)$ for the same reason, because $\fka \in \calX_B$  by Assertion (3). Remember that the element $c$ is chosen so that $I:_RJ = (c) +I$ (see the proof of Theorem \ref{2.9} (1)). We then have $\rmr(B/\fka)=\rmr(R/[I:_RJ])$, because $B = R + \fka$ and $$R/L=R/[\fka \cap R] \cong B/\fka$$ where the first equality follows from Lemma \ref{2.6} (2). Thus $$\rmr(B)=t{\cdot}\rmr(B/\fka)=t{\cdot}\rmr(R/L)=\rmr(R),$$ as is claimed. To see the equality $\rme(B) = \rme(R)$, enlarging the residue class field of $R$, we may assume that $R/\m$ is infinite. Choose an element $\alpha\in\m$ so that $(\alpha)$ is a reduction of $\m$. Hence $\alpha B$ is a reduction of $\m B$, while  $\m B$ is a reduction of $\n$, because $$\n A=(\m + cA)A=\m A = (\m B)A.$$ Therefore, $\alpha B$ is a reduction of $\n$, so that$$\rme(B)=\ell_B(B/\alpha B)=\ell_R(B/\alpha B)=\rme_{\alpha R}^{0}(B)=\rme_{\alpha R}^{0}(R)=\rme(R),$$ where the second equality follows from the fact that $R/\m \cong B/\n$ and the fourth equality follows from the fact that $\ell_R(B/R)<\infty$. Hence $\rme(B) = \rme(R)$ and $\rmr(B) = \rmr(R)$. Because $\rmv(R)=\rme(R)>1$ if and only if $\rmr(R)=\rme(R)-1$, the assertion that $\rmv(B)=\rme(B)$ if and only if $\rmv(R)=\rme(R)$ now follows. 
\end{proof}

We need one more lemma.

\begin{lemma}\label{3.5}
Suppose that $I, J\in \mathcal{X}_R$ and $I\subsetneq J$. Let  $\alpha\in J$. Then $J=(\alpha)+I$ if and only if $J^2=\alpha J.$
\end{lemma}

\begin{proof} It suffices to show the {\em only if} part. Suppose $J =(\alpha)+I$. We set $A = I:I$, $B=J:J$, and choose $b\in J$ so that $J^2=bJ$. Then $J=bB$ and $B \subseteq A$, whence $JA=bA$, while $JA=[(\alpha)+I]A=\alpha A+I$. We now choose $c\in\m$ so that $I=bcA$ (see Theorem \ref{3.1} (3)). We then have $bA=JA=\alpha A+bcA$, whence $bA=\alpha A$ by Nakayama's lemma. Therefore, $JA=\alpha A$, whence $(\alpha)$ is a reduction of $J$, so that $J^2=\alpha J$. 
\end{proof}

We are now ready to prove Theorem \ref{1.2}.

\begin{proof}[Proof of Theorem \ref{1.2}] 
Let $I, J \in \calX_R$ such that $I \subsetneq J$. We set $A = I:I$ and $B = J:J$. Let $b \in J$ such that $J = (b) + I$. Then $J^2 = bJ$ by Lemma \ref{3.5} and $B$ is a local ring with $\fkn = \fkm + \displaystyle\frac{I}{b}$ the maximal ideal by Theorem \ref{3.7}.

Let $\fka \in \calX_R$ such that $I \subseteq \fka \subsetneq J$. First of all, let us check the following.

\begin{claim}\label{claim1}
$\displaystyle\frac{\fka}{b} \in \calX_B$ and $\displaystyle\frac{\fka}{b} = \fka :J$.
\end{claim}

\begin{proof}[Proof of Claim \ref{claim1}] Since $b\in J$, $\fka:J \subseteq \displaystyle\frac{\fka}{b}$. On the other hand, since $$B=R:J \subseteq R:\fka=\fka:\fka$$ by Lemma \ref{2.1}, we get
$$J\cdot \frac{\fka}{b}=\fka\cdot \frac{J}{b}=\fka B\subseteq \fka {\cdot}(\fka:\fka)=\fka,$$ so that $\displaystyle\frac{\fka}{b}$ is an ideal of $B = \displaystyle\frac{J}{b}$ and $\fka:J=\displaystyle\frac{\fka}{b}$. Since $\displaystyle\frac{I}{b}\in\mathcal{X}_B$ by Theorem \ref{3.7} (3), to show that $\displaystyle\frac{\fka}{b} \in \calX_B$, we may assume $I \subsetneq \fka$. We then have,  by   Theorem \ref{1.1} (2), elements  $a_1, a_2\in\m$ such that $I =ba_1a_2A$ and $\fka = (ba_1)+I$; hence $\displaystyle\frac{\fka}{b}=a_1R+\displaystyle\frac{I}{b}$. We get $\displaystyle\frac{\fka}{b}=a_1B+\displaystyle\frac{I}{b}$, since $\displaystyle\frac{\fka}{b}$ is an ideal of $B$. Therefore, $\displaystyle\frac{\fka}{b}\in\mathcal{X}_B$ by Corollary \ref{3.6}, because $a_1a_2B$ is a reduction of $\displaystyle\frac{I}{b}=a_1a_2A$.
\end{proof}

We now have the correspondence $\varphi$ defined by $\fka \mapsto \displaystyle\frac{\fka}{b}$, and it is certainly injective. Suppose that $\fkb\in\mathcal{X}_B$ and $\displaystyle\frac{I}{b}\subsetneq\fkb$. We take $\alpha\in\fkb$ so that $\fkb^2=\alpha \fkb$. Then, since $B$ is a Cohen-Macaulay local ring with maximal ideal $\fkm + \displaystyle\frac{I}{b}$, we have $\fkb=\alpha B+\displaystyle\frac{I}{b}$ by Theorem \ref{3.1}. Let us write $\alpha=a+x$ with $a\in\m$ and $x\in\displaystyle\frac{I}{b}$. We then have $\fkb=aB+\displaystyle\frac{I}{b}$, so that $\fkb^2=a\fkb$ by Lemma \ref{3.5}. Set $L=\displaystyle\frac{I}{b}$. Then, since $A=I:I=L:L$, by Theorem \ref{3.1} we have an element $\beta\in\n=\m+L$ such that $L=a\beta A$; hence $a\beta \in L$. Let us write $\beta=c+y$ with $c\in\m$ and $y\in L$. We then have $ac =a\beta-ay \in L$ and $yA \subseteq L$, so that because $$L=a\beta A\subseteq acA+a{\cdot}yA\subseteq acA+\m L,$$ we get $L=acA$ by Nakayama's lemma. Therefore, $I=abcA$. On the other hand, since $aB=aR+a{\cdot}\displaystyle\frac{I}{b}$, we get $\fkb=aB + \displaystyle\frac{I}{b} = aR+\displaystyle\frac{I}{b}$. Hence, because $b\fkb=(ab)+I$ and $I=(ab)cA$, we finally have that $b \fkb \in \mathcal{X}_R$ and $$I=abcA \subsetneq b\fkb = (ab)+ I \subsetneq J$$ by Theorem  \ref{1.1} (1). Thus, the correspondence $\varphi$ is bijective, which completes the proof of Theorem \ref{1.2}.
\end{proof}

%%%%%%%%%%%%%%%%%%%%%%%%%%%%%%%%%%%%%%%%%%%%%%%%%%%%%%%%%%%%%%%%%%%%%%%%%%%%%%%%%%%%%%%%%%%%

\section{The case where $R$ possesses minimal multiplicity}

In this section, we focus our attention on the case where $R$ possesses minimal multiplicity. Throughout, we assume that $\rmv(R)=\rme(R)>1$. Hence, $\fkm \in \calX_R$ and $\mu_R(I) = v$ for all $I \in \calX_R$ by Corollary \ref{3.2}, where $v = \rmv(R)$.  We choose an element $\alpha\in\m$ so that $\m^2=\alpha\m$.

Let $I, J\in\mathcal{X}_R$ such that $I\subsetneq J$ and assume that there are no Ulrich ideals contained strictly between $I$ and $J$. Let $b\in J$ with $J^2=bJ$ and set $B=J:J$. Hence $B = \displaystyle\frac{J}{b}$, and $J=(b)+I$ by Theorem \ref{3.1}. Remember that by Theorem \ref{3.7}, $B$ is a local ring and $\rmv(B)=\rme(B) = \rme(R)>1$. We have $\n^2=\alpha\n$ by the proof of Theorem \ref{3.7} (4), where $\n$ denotes the maximal ideal of $B$.

We furthermore have the following.

\begin{lemma}\label{5.1}
The following assertions hold true.
\begin{enumerate}[{\rm (1)}]
\item $\ell_R(J/I)=1$.
\item $I=b\n=J\n$. Hence, the ideal $I$ is uniquely determined by $J$, and $I:I=\n:\n$. 
\item $(b\alpha)$ is a reduction of $I$. If $I = (b\alpha) + (x_2, x_3,\ldots, x_v)$, then $J=(b, x_2, x_3,\ldots, x_v)$.
\end{enumerate}
\end{lemma}

\begin{proof}
By Theorem \ref{1.2}, we have the one-to-one correspondence $$\{\fka\in\mathcal{X}_R |\  I\subseteq \fka\subsetneq J\} \overset{\varphi}{\longrightarrow} \{\fkb\in\mathcal{X}_B |\  \frac{I}{b}\subseteq \fkb\},\ \  \fka\mapsto \frac{\fka}{b},$$ where the set of the left hand side is a singleton consisting of $I$, and the set of the right hand side contains $\n$. Hence $\n=\displaystyle\frac{I}{b}$, that is $I=b\n=J\n$, because $J = bB$. Therefore, $I^2=b^2\n^2=b\alpha{\cdot}b\n=b\alpha{\cdot}I$, so that $(b\alpha)$ is a reduction of $I$. Because $$J/I=bB/b\n \cong B/\n$$ and $R/\m \cong B/\n$ by Theorem \ref{3.7} (2), we get $\ell_R(J/I)=1$. Assertion (3) is clear, since $J = (b)+I$.
\end{proof}

Since $\ell_R(R/I) < \infty$ for all $I \in \calX_R$, we get the following.

\begin{corollary}\label{5.2}
Suppose that $I, J\in\mathcal{X}_R$ and $I\subsetneq J$. Then there exists a composition series $I=I_\ell\subsetneq I_{\ell-1}\subsetneq\cdots\subsetneq I_1=J$ such that $I_i\in\mathcal{X}_R$ for all $1\leq i\leq\ell$, where $\ell=\ell_R(J/I)+1$.
\end{corollary}

The following is the heart of this section. 

\begin{thm}\label{5.3}
The set $\mathcal{X}_R$ is totally ordered with respect to inclusion.
\end{thm}

\begin{proof}
Suppose that there exist $I, J\in\mathcal{X}_R$ such that $I\nsubseteq J$ and $J\nsubseteq I$. Since $I\subsetneq\m$ and $J\subsetneq\m$, thanks to Corollary \ref{5.2}, we get composition series 
$$I=I_\ell\subsetneq I_{\ell-1}\subsetneq\cdots\subsetneq I_1=\m\ \ \ \text{and}\ \ J=J_n\subsetneq J_{n-1}\subsetneq\cdots\subsetneq J_1=\m$$
such that  $I_i, J_j\in\mathcal{X}_R$ for all $1\leq i\leq\ell$ and $1\leq j\leq n$. We may assume $\ell\leq n$. Then Lemma \ref{5.1} (2) shows that $I_i=J_i$ for all $1\leq i\leq\ell$, whence $J \subseteq J_\ell=I_\ell\subseteq I$. This  is a contradiction.
\end{proof}

\begin{remark}
Theorem \ref{5.3} is no longer true, unless $R$ possesses minimal multiplicity. For example, let $k$ be a field and consider $R = k[[t^3,t^7]]$ in the formal power series ring $k[[t]]$. Then, $\calX_R= \{(t^6-ct^7, t^{10}) \mid 0 \ne c \in k\}$, which is not totally ordered, if $\sharp k >2$. See Example \ref{4.6} (3) also.
\end{remark}

Let us now summarize the results in the case where $R$ possesses minimal multiplicity.

\begin{thm}\label{5.4}
Let $I\in\mathcal{X}_R$ and take a composition series $$(E)\ \ \ I=I_\ell\subsetneq I_{\ell-1}\subsetneq\cdots\subsetneq I_1=\m$$ so that $I_i\in\mathcal{X}_R$ for every $1\leq i\leq \ell=\ell_R(R/I)$. We set $B_0=R$ and $B_i=I_i:I_i$ for $1\leq i\leq\ell$ and let $\n_i=\rmJ(B_i)$ denote the Jacobson radical of $B_i$ for each $0\leq i\leq\ell$. Then we obtain a tower
$$R=B_0 \subsetneq B_1 \subsetneq \cdots \subsetneq B_{\ell -1} \subsetneq B_\ell \subseteq \overline{R}$$
of birational finite extensions of $R$ and furthermore have the following.
\begin{enumerate}[{\rm (1)}]
\item $(\alpha^i)$ is a reduction of $I_i$ for every $1\leq i\leq\ell$.
\item $B_i=\n_{i-1}:\n_{i-1}$ for every $1\leq i\leq\ell$.
\item For $0\leq i\leq\ell-1$, $(B_i,\n_i)$ is a local ring with $\rmv(B_i)=\rme(B_i)=\rme(R)>1$ and $\n_i^2=\alpha\n_i$.
\item Choose $x_2,x_3, \ldots, x_v\in I$ so that $I=(\alpha^\ell, x_2,\ldots, x_v)$. Then $I_i=(\alpha^i, x_2,x_3, \ldots, x_v)$ for every $1\leq i\leq\ell$. In particular, $\m = (\alpha, x_2, x_3, \ldots, x_v)$, so that the series $(E)$ is a unique composition series of ideals in $R$ which connects $I$ and $\m$. 
\item[{\rm (5)}] Let $J$ be an ideal of $R$ and assume that $I \subseteq J \subseteq \m$. Then $J = I_i$ for some $1 \le i \le \ell$.
\end{enumerate}
\end{thm}

\begin{proof}
The uniqueness of composition series in Assertion (4) follows from the fact that the maximal ideal $\m/I$ of $R/I$ is cyclic, and then, Assertion (5) readily follows from the uniqueness. Assertions (1), (2), (3), and the first part of Assertion (4) follow by standard induction on $\ell$. 
\end{proof}

\begin{corollary}\label{5.5}
Suppose that there exists a minimal element $I$ in $\mathcal{X}_R$. Then $\sharp\mathcal{X}_R=\ell<\infty$ with $\ell=\ell_R(R/I)$.
\end{corollary}

\begin{proof}
Since $\calX_R$ is totally ordered by Theorem \ref{5.3}, $I$ is the smallest element in $\calX_R$, so that $I \subseteq J$ for all $J \in \calX_R$. Therefore, by Theorem \ref{5.4} (5), $J$ is one of the $I_i$'s in the compoosition series $I=I_\ell\subsetneq I_{\ell-1}\subsetneq\cdots\subsetneq I_1=\m$.
\end{proof}

\begin{corollary}\label{5.5a}
If $\widehat{R}$ is a reduced ring, then $\calX_R$ is a finite set.
\end{corollary}

\begin{proof}
Since by Theorem \ref{5.4} $\ell_R(R/I) \le \ell_R(\overline{R}/R)<\infty$ for every $I \in \calX_R$, the set $\calX_R$ contains a minimal element, so that $\calX_R$ is a finite set.
\end{proof}

Here let us note the following.

\begin{ex}\label{5.6}
Let $(S, \n)$ be a two-dimensional regular local ring. Let $\n=(X, Y)$ and consider the ring $A=S/(Y^2)$. Then $\rmv(A) = \rme(A)=2$ and $$\mathcal{X}_A=\left\{(x^n, y)  \mid  n \geq 1\right\}$$ where $x, y$ denote the images of $X,Y$ in $A$, respectively. Hence $\sharp\mathcal{X}_A=\infty$.
\end{ex}

\begin{proof}
Let $I_n=(x^n, y)$ for each $n\geq 1$. Then $(x^n)\subsetneq I_n$ and $I_n^2=x^nI_n$. Let $\rmJ(A)=(x, y)$ be the maximal ideal of $A$. We then have $J(A)^2=xJ(A)$, whence $\rmv(A)=\rme(A)=2$. Because $I_n=(x^n):_Ay$, we get $I_n/(x^n)\cong A/I_n$. Therefore, $I_n\in\mathcal{X}_A$ for all $n \ge 1$. To see that $\calX_A$ consists of these ideals $I_n$'s, let $I\in\mathcal{X}_A$ and set $\ell=\ell_A(A/I)$. Then $I\subseteq I_\ell$ or $I\supseteq I_\ell$, since $\mathcal{X}_A$ is totally ordered. In any case, $I=I_\ell$, because $\ell_A(A/I_\ell)=\ell$. Hence $\mathcal{X}_A=\left\{(x^n, y)  \mid   n\geq1\right\}$.
\end{proof}

We close this section with the following. Here, the hypothesis about the existence of a fractional canonical ideal $K$ is equivalent to saying that $R$ contains an $\m$-primary ideal $I$ such that $I \cong \rmK_R$ as an $R$-module and such that $I$ possesses a reduction $Q=(a)$ generated by a single element $a$ of $R$ (\cite[Corollary 2.8]{GMP}). The latter condition is satisfied, once $\rmQ(\widehat{R})$ is a Gorenstein ring and the field $R/\m$ is infinite.

\begin{thm}\label{5.8}
Suppose that there exists a fractional ideal $K$ of $R$ such that $R\subseteq K\subseteq \overline{R}$ and $K\cong\rmK_R$ as an $R$-module. Then the following conditions are equivalent.
\begin{enumerate}[{\rm (1)}]
\item $\sharp\mathcal{X}_R=\infty$.
\item $\rme(R)=2$ and $\widehat{R}$ is not a reduced ring.
\item The ring $\widehat{R}$ has the form $\widehat{R} \cong S/(Y^2)$ for some regular local ring $(S,\n)$ of dimension two with $Y \in \n \setminus \n^2$.
\end{enumerate}
\end{thm}

\begin{proof}
$(1)\Rightarrow(2)$ The ring $\widehat{R}$  is not reduced by Corollary \ref{5.5a}. Suppose $R$ is not a Gorenstein ring; hence $R \subsetneq K$ and $\rme(R) > 2$. We set $\fka = R:K$. Let $I \in \calX_R$. Then, since $\mu_R(I)=v = \rme(R)>2$ by Corollary  \ref{3.2}, we have $\fka \subseteq I$ by \cite[Corollary 2.12]{GTT2}, so that $\ell_R(R/I)\leq\ell_R(R/\fka)<\infty$. Therefore, the set $\calX_R$ contains a minimal element, which is a contradiction.

$(3) \Rightarrow(1) $ See Example \ref{5.6} and use the fact that there is a one-to-one correspondence $I \mapsto I\widehat{R}$ between Ulrich ideals of $R$ and $\widehat{R}$, respectively. 

$(2) \Rightarrow (3)$ Since $\rmv(R)=\rme(R)=2$, the completion $\widehat{R}$ has the form $\widehat{R} = S/I$, where $(S,\n)$ is a two-dimensional regular local ring and $I=(f)$ a principal ideal of $S$. Notice that $\rme(S/(f))=2$ and $\sqrt{(f)} \ne (f)$. We then have $(f)=(Y^2)$ for some $Y \in \n \setminus \n^2$, because $f \in \n^2 \setminus \n^3$.
\end{proof}

\begin{remark}
In Theorem \ref{5.8}, the hypothesis on the existence of fractional canonical ideals $K$ is not superfluous. In fact, let $V$ denote a discrete valuation ring and consider the idealization $R = V \ltimes F$ of the free $V$-module $F=V^{\oplus n}$~($n \ge 2$). Let $t$ be a regular parameter of $V$. Then for each $n \ge 1$, $I_n = (t^n) \times F$ is an Ulrich ideal of $R$ (\cite[Example 2.2]{GOTWY1}). Hence $\calX_R$ is infinite, but $\rmv(R) = \rme(R) =n+1 \ge 3$.  
\end{remark}

Higher dimensional cases are much wilder. Even though $(R,\m)$ is a two-dimensional Cohen-Macaulay local ring possessing minimal multiplicity, the set $\calX_R$ is not necessarily totally ordered. Before closing this section, let us note examples.

\begin{ex} We consider two examples. 
\begin{enumerate}[{\rm (1)}]
\item Let $S = k[[X_0, X_1, \ldots, X_n]]$~($n \ge 3$) be the formal power series ring over a field $k$. Let $\ell \ge 1$ be an integer and consider the $2 \times n$ matrix $$\bbM = \begin{pmatrix}
X_1&X_2&\cdots&X_{n}\\
X_0^\ell&X_1&\cdots&X_{n-1}\\
\end{pmatrix}.$$ We set $R = S/\bbI_2(\bbM)$, where $\bbI_2(\bbM)$ denotes the ideal of $S$ generated by the $2 \times 2$ minors of the matrix $\bbM$. Then, $R$ is a Cohen-Macaulay local ring of dimension two, possessing minimal multiplicity. For this ring, we have $$\calX_R=\{(x_0^i, x_1, x_2, \ldots, x_n) \mid 1 \le i \le \ell\},$$ where $x_i$ denotes the image of $X_i$ in $R$ for each $0 \le i \le n$. Therefore, the set $\calX_R$ is totally ordered with respect to inclusion.

\item Let $(S,\n)$ be a regular local ring of dimension three. Let $F, G, H, Z \in \n$ and assume that $\n =(F,G, Z)=(G,H,Z)=(H,F,Z)$. (For instance, let $S=k[[X,Y,Z]]$ be the formal power series ring over a field $k$ with $\operatorname{ch} k \ne 2$, and choose $F = X, G=X+Y, H=X-Y$.) We consider the ring $R = S/(Z^2 - FGH)$. Then $R$ is a two-dimensional Cohen-Macaulay local ring of minimal multiplicity two. Let $f,g,h, z$ denote, respectively, the images of $F,G,H,Z$ in $R$. Then, $(f,gh, z)$, $(g, fh, z)$, $(h,fg,z)$ are Ulrich ideals of $R$, but any two of them are incomparable. 
\end{enumerate}
\end{ex}

\section{The case where $R$ is a $\GGL$ ring}
In this section, we study the case where $R$ is a $\GGL$ ring. The notion of $\GGL$ rings is given by \cite{GK}. Let us briefly review the definition.

\begin{definition}[\cite{GK}]\label{4.1}
Suppose that $(R,\m)$ is a Cohen-Macaulay local ring with $d= \dim R \ge 0$, possessing the canonical module $\rmK_R$.  We say that $R$ is a {\em generalized Gorenstein local} ($\GGL$ for short) ring, if one of the following conditions is satisfied.
\begin{enumerate}[{\rm (1)}]
\item $R$ is a Gorenstein ring.
\item $R$ is not a Gorenstein ring, but there exists an exact sequence
$$0 \to R \xrightarrow{\varphi} \rmK_R \to C \to 0$$
of $R$-modules and an $\m$-primary ideal $\fka$ of $R$ such that
\begin{enumerate}
\item[$\mathrm{(i)}$] $C$ is an Ulrich $R$-module with respect to $\fka$ and
\item[$\mathrm{(ii)}$] the induced homomorphism $R/\fka \otimes_R \varphi : R/\fka \to \rmK_R/\fka \rmK_R$ is injective.
\end{enumerate}
\end{enumerate}
When  Case (2) occurs, we especially say that $R$ is a $\GGL$ ring with respect to $\fka$.
\end{definition}

Since our attention is focused on the one-dimensional case, here let us  summarize a few results on $\GGL$ rings of dimension one. Suppose that $(R,\m)$ is a Cohen-Macaulay local ring of dimension one, admitting a fractional canonical ideal $K$. Hence, $K$ is an $R$-submodule of $\overline{R}$ such that $K \cong \rmK_R$ as an $R$-module and $R \subseteq K \subseteq \overline{R}$. One can consult  \cite[Sections 2, 3]{GMP} and \cite[Vortrag 2]{HK} for basic properties of $K$. We set $S = R[K]$ in $\rmQ(R)$. Therefore, $S$ is a birational finite extension of $R$ with $S=K^n$ for all $n \gg 0$, and the ring $S=R[K]$ is independent of the choice of $K$ (\cite[Theorem 2.5]{CGKM}). We set $\fkc=R:S$. First of all, let us note the following.

\begin{lemma}[{cf. \cite[Lemma 3.5]{GMP}}]\label{4.1a} $\fkc = K : S$ and $S=\fkc:\fkc=R:\fkc$.
\end{lemma}

\begin{proof} Since $R = K:K$ (\cite[Bemerkung 2.5 a)]{HK}), we have $\fkc = (K:K):S = K : KS = K : S$, while $R:\fkc = (K:K):\fkc = K: K\fkc = K : \fkc$. Hence $R : \fkc = K : \fkc = K:(K:S) = S$ (\cite[Definition 2.4]{HK}). Therefore, $\fkc:\fkc = (K:S):\fkc = K:S\fkc = K:\fkc = S$.
\end{proof}

We then have the characterization of $\GGL$ rings.

\begin{thm}[\cite{GK}]\label{4.2} Suppose that $R$ is not a Gorenstein ring. Then the following conditions are equivalent.
\begin{enumerate}[$(1)$]
\item $R$ is a $\GGL$ ring with respect to some $\fkm$-primary ideal $\fka$ of $R$.
\item $K/R$ is a free $R/\fkc$-module.
\item $S/R$ is a free $R/\fkc$-module.
\end{enumerate}
When this is the case, one necessarily has $\fka = \fkc$, and the following assertions hold true.
\begin{enumerate}[{\rm (i)}]
\item $R/\fkc$ is a Gorenstein ring.
\item $S/R\cong (R/\fkc)^{\oplus \rmr(R)}$ as an $R$-module.
\end{enumerate}
\end{thm}

The following result is due to \cite{GK, GTT2}. Let us include a brief proof of Assertion (1) for the sake of completeness.

\begin{thm}[\cite{GK, GTT2}]\label{4.3} Suppose that $R$ is not a Gorenstein ring.
Let $I\in\mathcal{X}_R$. Then the following assertions hold true.
\begin{enumerate}[{\rm (1)}]
\item If $I\subseteq \fkc$, then $I=\fkc$.
\item If $\mu_R(I)\neq 2$, then $\fkc\subseteq I$.
\item $\fkc\in\mathcal{X}_R$ if and only if $R$ is a $\GGL$ ring and $S$ is a Gorenstein ring.
\end{enumerate}
\end{thm}

\begin{proof}
(1)  Let $I \in \calX_R$ and assume that $I \subseteq \fkc$. We choose an element $a \in I$ so that $I^2=aI$. We then have $I \ne (a)$ and $I/I^2$ is a free $R/I$-module. Let $A = I:I$; hence $I = aA$. On the other hand, because $\fkc \subseteq I$, by Lemmata \ref{2.1} and  \ref{4.1a} we have $$A = R:I \supseteq R : \fkc = S \supseteq K.$$ 

\begin{claim}\label{claim2}
$A$ is a Gorenstein ring and $A/K$ is the canonical module of $R/I$.
\end{claim}
 
\begin{proof}[Proof of Claim \ref{claim2}]
Taking the $K$-dual of the canonical exact sequence
$0\to I \to R \to R/I \to 0$, we get the exact sequence
$$0 \to K \overset{\iota}{\to} K:I \to \Ext_R^1(R/I,K) \to 0,$$ where $\iota : K \to K:I$ denotes the embedding. On the other hand, $K:I=A$, because $$I = R:A=(K:K):A = K : KA = K: A$$
(remember that $K \subseteq A$). Therefore, since $I=K:A$ is a canonical ideal of $A$ (\cite[Korollar 5.14]{HK}) and $I=aA \cong A$, $A$ is a Gorenstein ring, and $A/K \cong \Ext_R^1(R/I,K)$.
\end{proof}

We consider the exact sequence $0 \to (a)/aI \to I/aI \to I/(a) \to 0$ of $R/I$-modules. Then, because $I = aA$, we get the canonical isomorphism between the exact sequences$$\xymatrix{
0\ar[r] &R/I\ar[r]^{i}\ar[d]^{\wr}&A/I\ar[r]\ar[d]^{\wr} &A/R\ar[r]\ar[d]^{\wr} &0\\
0\ar[r] &(a)/aI\ar[r]^{i}\ar@{}[ur]|{\circlearrowleft} &I/aI\ar[r]\ar[r]\ar@{}[ur]|{\circlearrowleft} &I/(a)\ar[r] &0
}$$
of $R/I$-modules, where $A/I$ is a Gorenstein ring, since $A$ is a Gorenstein ring and $I = aA$. Therefore, since $A/I ~(\cong I/aI)$ is a flat extension of $R/I$, $R/I$ is a Gorenstein ring, so that $A/K \cong R/I$ by Claim \ref{claim2}. Consequently, the exact sequence 
$$0 \to K/R \to A/R \to A/K \to 0$$ of $R/I$-modules is split, whence $K/R$ is a non-zero free $R/I$-module, because so is $A/R~(\cong I/(a))$. Hence, $\fkc = R:S \subseteq R:K =R:_RK = I$, so that $I = \fkc$. 
\end{proof}

Thanks to Theorem \ref{4.3}, we get the following.

\begin{thm}\label{4.4}
Let $R$ be a $\GGL$ ring and assume that $R$ is not a Gorenstein ring. Then the following assertions hold true.
\begin{enumerate}[{\rm (1)}]
\item $ \{ I \in \mathcal{X}_R |\  \fkc\subsetneq I \} =\{(a)+\fkc \ \mid \ a\in \m \ \text{such that}~ \fkc =abS ~\text{for~some}~b \in \m\}$.\\
In particular, $\fkc\in\mathcal{X}_R$, once the set $\{ I \in \mathcal{X}_R |\  \fkc\subsetneq I \}$ is non-empty.
\item $\mu_R(I)=\rmr(R)+1$ for all $I \in \calX_R$ such that $\fkc \subseteq I$.
\item $\{I \in\mathcal{X}_R |\ \fkc\subseteq I \}=\{I \in\mathcal{X_R} |\ \mu_R(I)\neq 2\}$.
\end{enumerate}
Therefore, if $R$ possesses minimal multiplicity, then the set $\calX_R$ is totally ordered, and $\fkc$ is the smallest element of $\calX_R$.
\end{thm}

\begin{proof}
(1) Let us show the first equality. First of all, assume that $\fkc\in\mathcal{X}_R$. Then since $S=\fkc:\fkc$, for each $\alpha\in\fkc$, $(\alpha)$ is a reduction of $\fkc$ if and only if  $\fkc=\alpha S$, so that the required equality follows from Corollary \ref{3.6}. Assume that $\fkc\notin \mathcal{X}_R$. Hence, by Theorem \ref{4.3} (3), $S$ is not a Gorenstein ring, because $R$ is a $\GGL$ ring. Therefore, since $\fkc = K:S$ is a canonical module of $S$ (Lemma \ref{4.1a} and \cite[Korollar 5.14]{HK}), we have $\fkc \ne \alpha S$ for any $\alpha\in\fkc$, whence the set $\{(a)+\fkc  \mid  a \in \m~\text{such that}~abS=\fkc~\text{for~some}~b \in \m\}$ is empty. On the other hand, since $S=\fkc :\fkc = R:\fkc$ and $S/R\cong (R/\fkc)^{\oplus \rmr(R)}$ (see Theorem \ref{4.2} (ii)), the $\m$-primary ideal $\fkc$ of $R$ satisfies Condition (C) in Definition \ref{2.2}. Therefore, if the set $\{ I\in \mathcal{X}_R \mid  \fkc\subsetneq I \}$ is non-empty, then $\fkc \in \mathcal{X}_R$ by Theorem \ref{2.9} (2), because $\rmr(R) \ge 2$. Thus, $\{ I \in \mathcal{X}_R \mid  \fkc\subsetneq I \}=\emptyset$.

(2) By Assertion (1), we may assume $\fkc\in\mathcal{X}_R$. Then, $\fkc = \alpha S$ for some $\alpha \in \fkc$, and therefore, $\mu_R(\fkc)=\rmr(R)+1$, since $\fkc/(\alpha) \cong S/R\cong (R/\fkc)^{\oplus \rmr(R)}$.  Thus, by Theorem \ref{3.1}, $\mu_R(I)=\mu_R(\fkc)=\rmr(R)+1$  for every $I \in \mathcal{X}_R$ with  $\fkc \subseteq I$. 

(3) The assertion follows from Assertion (2) and Theorem \ref{4.3} (3).

The last assertion follows from Assertion (3), since $\mu_R(I) = \rmv(R) > 2$ for every $I \in \calX_R$ (see Corollary \ref{3.2}). 
\end{proof}

Combining Theorems \ref{1.1} and \ref{4.4}, we have the following.

\begin{cor}\label{4.5}
Let $R$ be a $\GGL$ ring and assume that $R$ is not a Gorenstein ring.   Then the following assertions hold true.
\begin{enumerate}[{\rm (1)}]
\item Let $a_1,a_2, \ldots, a_n, b\in\m$~$(n\geq 1)$ and assume that $\fkc = a_1a_2\cdots a_nbS$. We set $I_i=(a_1a_2 \cdots a_i)+\fkc$ for each $1\leq i\leq n$. Then $\fkc \in \calX_R$ and $I_i \in\mathcal{X}_R$ for all $1 \le i \le n$, forming a chain $\fkc \subsetneq I_n \subsetneq I_{n-1} \subsetneq \ldots \subsetneq I_1$ in $\calX_R$.
\item Conversely, let $I_1, I_2, \ldots, I_n \in \mathcal{X}_R$~$(n \ge 1)$ and assume that $\fkc \subsetneq I_n \subsetneq I_{n-1} \subsetneq \ldots \subsetneq I_1$. Then $\fkc \in \calX_R$ and there exist elements $a_1,a_2, \ldots, a_n, b \in \m$ such that $\fkc =a_1a_2\cdots a_nbS$ and $I_i=(a_1a_2\cdots a_i)+\fkc$ for all $1\leq i\leq n$.
\end{enumerate}
\end{cor}

Concluding this paper, let us note a few examples of $\GGL$ rings.

\begin{ex}\label{4.6}
Let $k[[t]]$ be the formal power series ring over a field $k$. 
\begin{enumerate}[{\rm (1)}]
\item Let $H=\left<5,7,9, 13\right>$ denote the numerical semigroup generated by $5,7,9, 13$ and  $R=k[[t^5,t^7,t^9, t^{13}]]$ the semigroup ring of $H$ over $k$. Then, $R$ is a $\GGL$ ring, possessing $S = k[[t^3,t^5,t^7]]$ and $\fkc = (t^7,t^9, t^{10}, t^{13})$. For this ring $R$, $S$ is not a Gorenstein ring, and $\calX_R=\emptyset$.

\item Let $R=k[[t^4,t^9,t^{15}]]$. Then, $R$ is a $\GGL$ ring, possessing $S = k[[t^3,t^4]]$ and $\fkc = (t^9, t^{12}, t^{15})=t^9S$. For this ring $R$, $\calX_R=\{\fkc\}$.

\item Let $R=k[[t^6, t^{13}, t^{28}]]$. Then, $R$ is a $\GGL$ ring, possessing $S= k[[t^2,t^{13}]]$ and $\fkc = (t^{24},t^{26}, t^{28})=t^{24}S$. For this ring $R$, the set
$\{I \in \calX_R \mid \fkc \subsetneq I\}$ 
consists of the following families.
\begin{enumerate}[{\rm (i)}]
\item $\{(t^6+ at^{13}) + \fkc \mid a \in k\}$,
\item $\{(t^{12} + at^{13} + bt^{19}) + \fkc \mid a, b \in k\}$, and
\item $\{(t^{18}+ at^{25}) + \fkc \mid a \in k\}$.
\end{enumerate}
For each $a \in k$, we have a maximal chain 
$$\fkc \subsetneq (t^{18}+at^{25})+\fkc \subsetneq (t^{12}+at^{19})+\fkc \subsetneq (t^{6}+at^{13})+\fkc$$
in $\calX_R$. On the other hand, for $a,b \in k$ such that $a \ne 0$,
$$\fkc \subsetneq (t^{12}+at^{13}+ bt^{19})+\fkc$$ is also a maximal chain in $\calX_R$.
\item Let $H = \left<6, 13, 28\right>$. Choose  integers $0 < \alpha  \in H$ and $1< \beta \in \Bbb Z$ so that  $\alpha \not\in \{6, 13, 28\}$ and $\GCD(\alpha, \beta)=1$. We consider $R= k[[t^\alpha, t^{6\beta}, t^{13\beta}, t^{28\beta}]]$. Then, $R$ is a $\GGL$ ring with $\rmv(R) = 4$ and $\rmr(R)=2$. For this ring $R$, $S=k[[t^\alpha, t^{2\beta}, t^{13\beta}]]$, and $\fkc = t^{24\beta}S$. For instance, take $\alpha = 12$ and $\beta = 5n$, where $n > 0$ and $\GCD(2,n)=\GCD(3,n)=1$. Then, $\fkc = t^{120n}S =(t^{12})^{10n}S$, so that the set $\{I \in \calX_R \mid \fkc \subsetneq I\}$ seems rather wild, containing  chains of large length.
\end{enumerate}
\end{ex}

%\addcontentsline{toc}{section}{references}

\end{document}